\documentclass[12pt]{article}
\usepackage{amsfonts, amsmath, amstext, amssymb}

\usepackage[dvips]{graphicx}
\usepackage{latexsym}

\usepackage{amscd}
\usepackage{latexsym}
\usepackage{amsthm}
\input{xy}
\xyoption{all}

\newtheorem{thm}{Theorem}[section]
\newtheorem{theorem}[thm]{Theorem}
\newtheorem{prop}[thm]{Proposition}
\newtheorem{proposition}[thm]{Proposition}

\newtheorem{lemma}[thm]{Lemma}

\newtheorem{cor}[thm]{Corollary}
\newtheorem{corollary}[thm]{Corollary}

\newtheorem{remark}[thm]{Remark}
\newtheorem{ex}[thm]{Example}

\newtheorem{definitiontemp}[thm]{Definition}
\newenvironment{defn}{\begin{definitiontemp}
\normalfont}{\end{definitiontemp}}

\usepackage{color}

\newenvironment{pf}{\begin{trivlist}\item[\hskip\labelsep
{\it Proof.}]}{\end{trivlist}}

\newenvironment{pftitle}[1]{\begin{trivlist}\item[\hskip\labelsep
{\it #1.}]}{\end{trivlist}}

\newcommand\commutativesquare[8]{
\xymatrix{#1\ar[d]_{#4}\ar[r]^{#2}& #3\ar[d]^{#5}\\
#6\ar[r]_{#7}&#8}
}

\newcommand{\Ftilde}{\tilde{F}}
\newcommand{\Q}{\mathbb{Q}}

\newcommand{\A}{\mathcal{A}}
\newcommand{\B}{\mathcal{B}}
\newcommand{\C}{\mathcal{C}}

\newcommand{\M}{\mathcal{M}}

\def\res{\!\!\upharpoonright\!}

\def\phi{\varphi}

\newcommand{\la}{\langle}
\newcommand{\ra}{\rangle}

\newcommand{\at}{\char'100}

\newcommand{\ktilde}{\tilde{k}}
\newcommand{\gtilde}{\tilde{g}}

\newcommand{\xtilde}{\tilde{x}}
\newcommand{\ytilde}{\tilde{y}}

\newcommand{\set}[2]{\ensuremath{ \{ #1 : #2 \} }}
\newcommand{\DS}[2]{\ensuremath{\text{DgSp}_{\mathcal{#1}}(#2)}}

\newcommand{\bfd}{\boldsymbol{d}}
\newcommand{\bfz}{\boldsymbol{0}}

\begin{document}

\title{Computably Categorical Fields via Fermat's Last Theorem}
\author{
Russell Miller\thanks{The first author was
partially supported by NSF grant 
\#DMS - 1001306, by several grants from
The City University of New York PSC-CUNY Research Award Program
and the Queens College Research Enhancement Fund,
and by grant \#13397 from the Templeton Foundation.
Both authors were supported by CUNY Collaborative
Grant \# 80209-04-12, and wish to acknowledge
useful conversations with Gunther Cornelissen,
Oleg Kudinov, Alexandra Shlapentokh, and Lucien Szpiro.
} \&
Hans Schoutens\thanks{The second author was
partially supported by NSF grant \# DMS-0500835,
and by grants numbered 67195-00 36, 68093-00 37,
69135-00 38, and 61085-00 39 from The City University
of New York PSC-CUNY Research Award Program.}}
\maketitle

\begin{abstract}
We construct a computable,
computably categorical field of infinite transcendence
degree over $\Q$, using the Fermat polynomials and
assorted results from algebraic geometry.
We also show that this field has an intrinsically
computable (infinite) transcendence basis.
\end{abstract}

\section{Introduction}
\label{sec:intro}

In computable model theory, we investigate
the extent to which various model-theoretic
constructions can be performed effectively.
For instance, given two structures, model theorists
naturally wish to consider isomorphisms between them.
In computable model theory,
we break this down into two problems.
First we pose the \emph{Isomorphism Problem},
in which we ask how difficult it is to determine in general
whether two given structures are isomorphic to each other at all.
If they are indeed isomorphic, then we ask about
the difficulty of actually computing an isomorphism
between them.  This latter question involves
the notion of \emph{computable categoricity}.
When asking these questions, we usually assume that the
structures themselves are computable,
meaning that we can compute the functions and relations
on them.  If they are in fact isomorphic,
then it is reasonable to ask whether there
exists a computable isomorphism between them.

Fields were the first mathematical structures
for which the notion of computable categoricity arose.
The isomorphism problem for fields is addressed
by Calvert in \cite{C04}.  Long before that, though,
in \cite{FS56}, Frohlich and Shepherdson had begun to
consider the second question, by giving an
example of (in their terminology; see their Corollary 5.51)
two isomorphic, explicitly presented fields with no explicit isomorphism
between them.  This idea eventually grew into the following definition.
\begin{defn}
\label{defn:cc}
A computable structure $\A$ is \emph{computably categorical}
if for every computable structure $\B$ isomorphic
to $\A$, there exists a computable isomorphism
from $\A$ onto $\B$.
\end{defn}
From the point of view of our second question, such a structure
is about as nice as it could be:  no matter which
two computable copies of $\A$ we choose,
there must exist a computable isomorphism between them.
(An even nicer version,
called \emph{uniform computable categoricity}
and examined in \cite{DHK03},
requires not only that a computable isomorphism
exist, but that we should be able to figure out
a program for computing it, just given the ability
to compute the structures $\A$ and $\B$.)
If the structure is not computably categorical,
we may ask how strong an oracle is required
to compute isomorphisms between two computable copies;
for a consideration of this question for algebraic
fields, for example, see \cite{M09}.

Much research has been devoted to characterizing
the computably categorical models of various theories,
including work by Dzgoev, Goncharov, Khisamiev,
Lempp, McCoy, Miller, Remmel, and Solomon.
Some results are readily stated:  we know that a computable
linear order is computably categorical if and only if it has
only finitely many pairs of consecutive elements,
for example, and that a computable Boolean algebra
is computably categorical if and only if it has finitely many atoms.
On the other hand, the known structural characterization
of computably categorical trees requires a description
by recursion on the heights of finite trees.
The question has been studied
for a number of other theories as well, and
results along these lines may be found in
\cite{G75}, \cite{G82}, \cite{G98}, \cite{GD80},
\cite{GLS03}, \cite{KS97},
\cite{LMMS05}, \cite{M05}, \cite{R81a}, and \cite{R81b}.

However, the original problem of computable categoricity for fields
has defied all attempts at structural characterization.
The most obvious conjecture would be that
the transcendence degree of a field over its
prime subfield should determine computable categoricity.
For algebraically closed fields, this is indeed
the case, as shown by Ershov in \cite{E77}:
an ACF is computably categorical if and only if it has
finite transcendence degree over its prime subfield.
However, in the same work, Ershov built
a field, algebraic over its prime subfield but not algebraically closed,
which was not computably categorical,
and a number of further results for algebraic fields
have been developed in \cite{MS12} by Shlapentokh and one of us,
who then extended them in \cite{HKMS12} jointly with Hirschfeldt and Kramer.
In this paper we refute the converse as well,
by constructing a computably categorical field of infinite
transcendence degree over the rationals $\Q$.
Thus, neither implication in the original
conjecture actually holds.

The counterexample $F$ we build is readily described.
It begins with an infinite, purely transcendental
extension $\Q(x_0,x_1,\ldots)$ of $\Q$.  Then, for each $i$,
we adjoin an element $y_i$ such that $(x_i,y_i)$ is a solution
of the Fermat curve $D_{p_i}$ defined by
$X_i^{p_i}+Y_i^{p_i}=1$, for an odd prime $p_i$.
Thus each $y_i$ is algebraic over $x_i$, but the infinite
set $\set{x_i}{i\in\omega}$ is still algebraically independent.
Fermat's Last Theorem shows that each $D_{p_i}$ has only
the trivial solutions $(0,1)$ and $(1,0)$ in $\Q$,
and the heart of our proof is a demonstration that
there exists a computable sequence $p_0<p_1<\cdots$ such that
in $F$, every nontrivial solution of the equation of $D_{p_i}$
generates the same subfield, namely $\Q(x_i,y_i)$.
Therefore, mapping $x_i$ and $y_i$ to any nontrivial
solution of $D_{p_i}$ in a field $\Ftilde$ isomorphic to $F$
will define an isomorphism.  The algebraic geometry
required is developed in Section \ref{sec:AG},
and the sequence $\la p_i\ra$ is chosen and the
field presented in Section \ref{sec:ccfields}.
In Section \ref{sec:basis}, we use this construction
to give the first example of a computable field possessing
an infinite, intrinsically computable transcendence basis.

We wish to make note here of unpublished work by Kudinov
and Lvov.  Working jointly, they addressed the same question,
about computable categoricity for fields of infinite
transcendence degree, and made significant progress on it.
Like us, they combined techniques from algebraic geometry and
computability theory, but their investigations were
unfortunately cut short when Lvov passed away, and it has
not been possible to reconstruct their work.  As we
understand it, their constructions did not make use of
the Fermat polynomials -- which suggests that there
are alternative ways to approach this problem, awaiting discovery
(or re-discovery) by an enterprising researcher.
We salute the efforts of Kudinov and Lvov, and regret the
demise of the latter.

We describe our principal conventions for this paper.
A \emph{computable field} is a structure
in the signature with addition and multiplication,
whose domain is an initial segment of $\omega$,
and for which those two operations are computable
and define a field.  This conforms to the usual
computable-model-theoretic definition of a computable structure.
An introduction to such fields for non-logicians
is given in \cite{Notices08} and continued in \cite{Proceedings08}.
Classic references on computable fields include
\cite{E77}, \cite{EG00}, \cite{FJ86}, \cite{FS56}, \cite{K1882},
\cite{MN79}, \cite{R60}, and \cite{ST99}.
Our computability-theoretic notation is standard
and can be found in \cite{S87}.

The field $\Q$
is known to be computably categorical, and so we will often
just write $\Q$ to denote a computable presentation
of that field, without concern for the specifics of the presentation.
Given a computable field $F$, its
polynomial ring $F[X]$ may be viewed just as
the set $F^*$ of finite tuples of elements of $F$,
with $\la a_0,\ldots,a_d\ra$ identified with $\sum a_i X^i$.
(For a perfect identification, ensure that if $a_d=0$, then $d=0$.)
Iterating this process yields a computable presentation
of the ring $F[X_1,\ldots, X_n]$, uniformly in $n$.

Given a computable field $F$, we will
treat any singly-generated field extension $F(x)$
as a computable field as well.  To compute it,
we will need to know whether $x$ is algebraic
over $F$ or not, and if it is, we will need
its minimal polynomial $p(X)$ over $F$.
In the algebraic case, one views elements of $F(x)$
as $F$-linear combinations over the set
$\{ 1,x,x^2,\ldots,x^{d-1}\}$, where $d=\deg(p)$,
with the obvious addition and multiplication
(which requires knowledge of $p(X)$, of course).
In the transcendental case, $F(x)$ is just the quotient field
of the domain $F[X]$ given above, and this
quotient field is computably presentable
as the Cartesian product $F[x]\times (F[x]-\{ 0\})$
modulo a computable equivalence relation.  We can iterate these
extensions, even over infinitely many generators,
as long as the minimal polynomial (or lack thereof)
for each generator over the preceding ones is given
effectively.  Notice that the base field $F$ is
a computable subfield of each extension built this way.

\newcommand\spec{\operatorname{Spec}}
\newcommand\tensor\otimes
\newcommand\iso\cong

\section{Results from algebraic geometry}
\label{sec:AG}

We now introduce some notation, review some
 algebraic geometry, and prove the results
 from this topic which will be needed in the subsequent sections.
 Since these results are nontrivial, we will assume
 a significant algebraic background; full explanation
 would require much more space.  Hartshorne's book \cite{H77}
 provides a comprehensive description of the subject.

 Let $k$ be a field of characteristic zero and fix  an algebraic closure
$\bar k$ of $k$.
 A
\emph{variety} $V$ over $k$ is by definition an absolutely irreducible,
separated, reduced scheme of
finite type over $k$. Recall that $V$ is called \emph{absolutely irreducible}
if the base change $V\times_k\bar k$, (or more correctly, the fiber
product $V\times_{\spec k}\spec \bar k$) is irreducible.
In fact, $V\times_kK$ is then irreducible for any extension $K$ of $k$.

Suppose first that $V$ is \emph{affine}, that is to say, a
closed
subscheme of affine space $\mathbb A_k^n$ for some $n$, or equivalently, a
scheme of the form $V=\spec A$ with $A$  an absolutely irreducible (meaning
that $A\tensor_k\bar k$ is a domain), finitely generated   $k$-algebra. For any extension
field $K$ of $k$, we define the \emph{$K$-rational points} of $V$, denoted
$V(K)$, as the collection of all $P\in \mathbb A_K^n(K)=K^n$ such that $P$ lies
on
$V$ (more precisely, in scheme-theoretic terms, this means that the morphism
$\spec
K\to \mathbb A_K^n$ induced by $P$  factors through $V$, yielding a morphism   $\spec K\to
V$; or equivalently, that $P$ is given by a maximal ideal
$\mathfrak m_P$ in $A$ such that $A/\mathfrak m_P\iso K$).
We define the
\emph{coordinate ring $k[V]$ of $V$ over $k$} as the residue ring
$k[x_1,\dots,x_n]/I_V$, where $x_i$
are variables and $I_V$ is the ideal of all polynomials vanishing on $V$ (that
is to say, the collection of all $f\in k[x_1,\dots,x_n]$ such that $f(P)=0$ for
all
$P\in V(\bar k)$).
Hence, in the above notation,
$A\iso k[V]$.
Since we assumed $V$ to be irreducible, $I_V$ is a prime
ideal and remains so when extended to $\bar k[x_1,\dots,x_n]$; we call such an
ideal \emph{absolutely prime}.
Conversely, if $I$ is an absolutely prime ideal in
 $k[x_1,\dots,x_n]$, the
\emph{variety defined by $I$} is the scheme
$V:=\spec (k[x_1,\dots,x_n]/I)$; for any extension $k\subseteq
K$, its
$K$-rational points
are precisely those $P\in K^n$ such that $f(P)=0$ for all $f\in I$.
This establishes a one-one
correspondence between (affine) subvarieties of $\mathbb A_k^n$
and absolutely prime ideals in $k[x_1,\dots,x_n]$ (which in turn are
in one-one correspondence with the %class of all 
absolutely irreducible, finitely generated $k$-algebras).

We define the   \emph{function field   $k(V)$ of $V$ over $k$} as the
field of fractions of $k[V]$. The function field of a variety is a so-called
\emph{birational invariant}, meaning that it only depends on the birational
class of $V$. In particular, by resolution of singularities in characteristic
zero, we may therefore
assume, when studying the function field, that the variety has no
singularities. The following quantities are all equal and are called
the \emph{dimension} of $V$:
\begin{itemize}
\item the transcendence
degree of $k(V)$ over $k$;
\item the least number of \emph{hypersurfaces} (=variety
defined by a single, absolutely irreducible equation) $H_1,\dots, H_d$ such
that $V(\bar k)\cap H_1(\bar k)\cap\dots\cap H_d(\bar k)$ is non-empty and finite;
\item the combinatorial dimension of $V(\bar k)$ viewed as a topological space
via its Zariski topology;
\item the Krull dimension of $k[V]$.
\end{itemize}
A point $P$ on $V$ is called \emph{generic over $k$} if the field $k(P)$
generated by
its coordinates is isomorphic to the function field of $V$, or equivalently,
if the
transcendence degree of $k(P)$  over $k$ is equal to the dimension of $V$.
Conversely, any point $P=(p_1,\dots,p_n)\in  K^n$ in some extension field
$K\supseteq k$ can be viewed
as the generic point of an affine variety $V$ over $k$:  let $\mathfrak m_P$
be the maximal ideal in $K[x_1,\dots,x_n]$ generated by the linear polynomials
$X_i-p_i$, let $I_V$ be the contraction of this ideal to
$k[x_1,\dots,x_n]$, and let $V$ be the affine variety with coordinate ring
$k[V]:=k[x_1,\dots,x_n]/I_V$. One verifies that $k(V)$ is isomorphic with the
field $k(P)$, that is to say, $P$ is a generic point of $V$.   If $V$ is not
affine, then its function field
can still be defined as the function field $k(V_0)$ of any non-empty  affine
open subset $V_0$ of $V$. In particular, when studying the function field, we
may take the variety to be projective as well. On occasion, we will use the
following simple observation:

\begin{lemma}\label{L:ac}
If $V$ is a variety over $k$, then $k$ is algebraically closed inside $k(V)$.
\end{lemma}
\begin{proof}
By assumption $V\times_k\bar k$ is irreducible, showing that
$k(V)\tensor_k\bar k$ is equal to the function field $\bar k(V)$ of $V$ over
$\bar k$.
Towards a contradiction, assume $k\subsetneq l\subseteq k(V)$ is
a finite extension of degree $d>1$. Hence $l\tensor_k\bar k\iso \bar
k^d$ is not a field,
which contradicts that $l\tensor_k\bar k\subseteq  k(V)\tensor_k\bar k$.
\end{proof}

\subsection{Curves}
By a  \emph{curve} over $k$ we will mean in this article a non-singular (also
called \emph{smooth}),
projective one-dimensional variety $C$ over $k$. Recall that for any one-dimensional
variety $X$ defined over $k$, there exists   a unique curve $C$ to which $X$ is
birationally equivalent, and moreover, if $X$ has no singularities, then it is
isomorphic to an open subset of $C$.
For instance, one can define a scheme structure on the
set $C_X$ of discrete valuations of $k(X)$ which are trivial on $k$, and then
show that $C_X$ is a curve birational to $X$ (see for
instance \cite[Theorem 7.3.1]{K93} or \cite[I. Corollary 6.11]{H77});
alternatively,
we can take a completion $\bar X$ of $X$ (e.g., the Zariski closure of $X$
viewed as a subvariety of some projective space over $k$), and then take its
normalization. A one-dimensional variety over $k$ which is birational to a curve
$C$ over $k$ is sometimes called a \emph{model} of $C$. Recall (see for
instance \cite[IV. Corollary 3.11]{H77}) that any curve admits an affine
plane model (with at most nodes as singularities). By Lemma~\ref{L:ac}, a point
$P$ on $C$ is generic (over $k$) if and only if $P\notin C(\bar k)$.

One can associate to any curve $C$ a natural number $g(C)$, called its
\emph{genus}, which is a birational invariant and does not change when
extending $k$ to a larger base field.
In case $C$ is a plane curve (that is, a curve in $\mathbb P_k^2$),
then its genus is calculated as $g(C):=(d-1)(d-2)/2$, where
$d$ is the degree of the defining equation of $C$ (see for instance \cite[I. Exercise 7.2]{H77}).
%The
%\emph{genus}
% of a curve $C$ can be
%defined as the number $1-H_C(0)$, where $H_C$ is the \emph{Hilbert polynomial} of
%$C$, or equivalently, by Serre duality, as the vector space dimension over $k$ of
%the global sections
% of the canonical sheaf $\omega_C$ of $C$ (see for instance
%\cite[\S IV.1]{H77}). Note that the value of the genus does not change when
%making a base change. In case $C$ is a plane curve (that is, a curve in $\mathbb P_k^2$),
%then its genus is calculated as $g(C):=(d-1)(d-2)/2$, where
%$d$ is the degree of the defining equation of $C$ (see for instance \cite[I. Exercise 7.2]{H77}).

Any morphism between curves $C\to D$ defined over $k$  is
either
\emph{constant} (that
is to say, its image is a single point, necessarily $k$-rational), or otherwise
\emph{surjective}.
In the latter case,
we have an inclusion of
function fields
$k(D)\subseteq k(C)$, which is necessarily a finite
extension, since both $k(C)$ and $k(D)$ have transcendence degree one over $k$.
It follows that $C\to D$
 is a \emph{finite} morphism of
\emph{degree}
$d:=((k(C):k(D))$; in particular, all the fibers of $C\to
D$ are   finite, of cardinality at most $d$.

%A morphism between curves $C\to D$ defined over $k$  is
%either
%\emph{constant} (that
%is to say, its image is a single point, necessarily $k$-rational), or otherwise
%\emph{surjective}.
%Indeed,
%take non-empty affine open   subsets $C_0$ and $D_0$ of $C$ and $D$
%respectively so that $C\to D$ induces a $k$-algebra
%homomorphism $k[D_0]\to k[C_0]$. The map $C\to D$ is not dominant precisely when
%this homomorphism has a non-zero kernel, which then by necessity has to be a
%maximal ideal $\mathfrak m$ of $k[D_0]$. Hence we have inclusions of fields
%$k\subseteq
%k[D_0]/\mathfrak m\subseteq k(C)$. By the Nullstellensatz, the former
%extension is algebraic, whence trivial by Lemma~\ref{L:ac}, showing that the
%image of $C\to D$ is the $k$-rational point on $D$ determined by $\mathfrak m$.
%On the other hand, if $C\to D$ is dominant, then it must be surjective: since
%$C$ and $D$ are projective, $C\to D$ is proper whence in particular has a
%closed image. Since this image is also dense, it must be equal to $D$.
%Moreover, in the surjective case,
%we have an inclusion of
%function fields
%$k(D)\subseteq k(C)$, which is necessarily a finite
%extension, since both $k(C)$ and $k(D)$ have transcendence degree one over $k$.
%It follows that $C\to D$
% is a \emph{finite} morphism of
%\emph{degree}
%$d:=((k(C):k(D))$; in particular, all the fibers of $C\to
%D$ are   finite, of cardinality at most $d$. 
The following fact will be
quite useful.

\begin{prop}\label{P:mor}
Let $C$ be a curve and $V$ any variety, both defined over a field $k$.
We have a functorial bijection of sets
$$ \operatorname{Mor}_k(C,V) \iso V(k(C)), $$
and under this bijection, the morphism $C\to V$ is constant if and only if
the associated $k(C)$-rational point of $V$ is   $k$-rational.  
\end{prop}
\begin{pf}
Let $P$ be a
$k(C)$-rational point on $V$, which therefore corresponds to
a morphism $\spec k(C)\to V$. By the valuative criterion for properness
(see for instance \cite[I. Proposition 6.8 or II. Theorem 4.7]{H77} or
\cite[Proposition 7.2.3]{K93}), this extends to a morphism $C\to V$.
Conversely given a morphism $v\colon C\to V$, choose affine open sets $C_0\subseteq C$
and $V_0\subseteq V$ such that $v(C_0)\subseteq V_0$. Hence we have an induced
$k$-algebra homomorphism $k[V_0]\to k[C_0]$. Composing this with the inclusion
$k[C_0]\subseteq k(C_0)=k(C)$ then yields a $k(C)$-rational point
$\spec k(C)\to V$.

Functoriality here means that we can view $\operatorname{Mor}_k(\cdot,V)$
and $V(\cdot)$ as contravariant functors, and the above bijection is compatible
with these functors in the following sense. Given a non-constant morphism $C\to D$,
composition   yields a map $\operatorname{Mor}_k(D,V)\to \operatorname{Mor}_k(C,V)$.
Moreover, we have an extension $k(D)\subset k(C)$ of function fields, giving rise
to an inclusion $V(k(D))\subset V(k(C)$. One now easily checks that we have a
commutative diagram
$$
\commutativesquare{\operatorname{Mor}_k(D,V)}{\iso}{V(k(D))}{}{}{
\operatorname{Mor}_k(C,V)}\iso {V(k(C))}
$$
Moreover, these bijections are  also compatible when viewed as (covariant) functors
in their second component, that is to say, given a non-constant morphism $V\to W$,
composition yields $\operatorname{Mor}_k(C,V)\to  \operatorname{Mor}_k(C,W)$,
and we get a natural map $V(k(C))\to W(k(C))$, making  the diagram
$$
\commutativesquare{\operatorname{Mor}_k(C,V)}{\iso}{V(k(C))}{}{}{
\operatorname{Mor}_k(C,W)}\iso {W(k(C))}
$$
commute as well.
\qed\end{pf}

Moreover, if $V$ is actually a curve, then by our above discussion,
under this isomorphism, the morphism $C\to V$ is constant if and only if
the associated $k(C)$-rational point of $V$ is in fact $k$-rational.

Given two curves $C$ and $D$ over $k$, we say that $C$
\emph{covers} (or \emph{dominates}) $D$ if there exists a non-constant morphism $C\to D$ defined over some field extension $K\supseteq k$.
Note that such a $C\to D$ is then automatically finite and surjective.
By Proposition \ref{P:mor}, therefore, $C$ does not cover  $D$ if and only if
$\operatorname{Mor}_K(C,D) \iso D(K)$ for every $K\supseteq  k$.

\begin{lemma}\label{L:noncov}
Let $C$   and   $D_i$ be  curves over $k$ such that no $D_i$ covers $C$.
If $F$ is the field
generated by all the function fields $k(D_i)$, then $C(F)=C(k)$.
\end{lemma}
\begin{proof}
The essence of the proof is to build $F$ not all at once from the fields $k(D_i)$,
but rather by induction.
Let $F_i$ be the subfield of $F$  generated by all $k(D_j)$ with $j\leq i$, that
is to
say, $F_i$ is defined inductively as the function field $F_{i-1}(D_i)$ of the curve $D_i$ viewed as a curve over $F_{i-1}$ (with $F_0:=k$). It suffices
to show by induction on $i$ that $C(F_i)=C(k)$. Suppose $P$
is an $F_i$-rational point on $C$ , and we want to show that $P$ is
$k$-rational. Since $P\in C(F_{i-1}(D_i))$, we get from Proposition \ref{P:mor}
a morphism $D_i\to C$ defined over $F_{i-1}$. Since $D_i$ does not cover $C$,
this morphism must be constant, showing that $P$ belongs to $C(F_{i-1})$ and
by induction, the latter is just $C(k)$.
\end{proof}

We will also use the following well-known
inequality (see for instance \cite[Corollary 2.4]{H77}),
which leads to a quick proof
of Proposition \ref{P:gen2} under the assumption
that the genus is at least $2$.  (Our construction in Section \ref{sec:ccfields}
will require a genus $\geq 2$ in any case.) 

\begin{theorem}\label{T:Hur}
Let $k$ be an algebraically closed field of characteristic zero and let $C$
and $D$ be curves of genus $g(C)$ and $g(D)$ respectively. If $C\to D$ is a
finite morphism of degree $d$, then $g(C)-1\geq d\cdot (g(D)-1)$.
\end{theorem}
\begin{proof}
Since the characteristic is zero, the morphism is separable, so that Hurwitz's Formula  yields
$$
2(g(C)-1)= 2d(g(D)-1)+\operatorname{deg}(R),
$$
 where $R$ is the
(effective) \emph{ramification divisor} of $C\to D$; see for instance \cite[IV.
Corollary 2.4]{H77}.
\end{proof}

\begin{proposition}\label{P:gen2}
Let $k$ be a field of characteristic zero and let $C$ be a curve
over
$k$ of genus $g\geq 2$. Then the function field $K:=k(C)$ of $C$ is generated by
the
coordinates of any $K$-rational  point $P$ of $C$ which is not $k$-rational,
that is to say, for any $P\in C(K)\setminus C(k)$, the natural inclusion
$k(P)\subseteq K$ is an equality.
\end{proposition}
\begin{proof}
Let
$C_0$ be a (non-empty) affine open subset of $C$ and let
$A:=k[C_0]$ be its coordinate ring.
For instance, if $C_0$ is a
plane curve with affine equation $f=0$, then $A=k[x,y]/fk[x,y]$
and $K$ is the  field of fractions of $A$. We need to show that if
$P:=(a,b)\in K^2\setminus k^2$ satisfies $f(a,b)=0$, then $k(P):=k(a,b)$ is equal
to $K$.
In
any case, since $P$ is not defined over $k$, it is a generic point of $C$ by
Lemma~\ref{L:ac}. In particular,  $k(P)$ and $K$ are isomorphic over $k$ (for
instance, in the planar example, the $k$-algebra $k[a,b]$ is easily
seen to be $k$-isomorphic to $A$). In particular,  by our discussion of
Proposition \ref{P:mor}, the embedding $k(P)\subseteq K$ is finite
of degree $d:=(K:k(P))$ inducing  a finite morphism $C\to C$.
We need to show that $d=1$.

Assume first that $k$ is algebraically closed, so that we
can apply Theorem \ref{T:Hur} to this morphism, giving  $g-1\geq d(g-1)$. Since
$g\geq 2$, we must
have $d=1$, as we wanted to show. For the general case, it follows from the
algebraically closed case that there
exists a finite extension $l$ of $k$ such that $l(P)=l(C)$. Let $\pi$ be such
that $l=k(\pi)$, so that in particular $l(P)=k(P,\pi)$ and $l(C)=K(\pi)$.
However,
since $\pi$ is algebraic over $k$ whence over the $k$-isomorphic fields $k(P)$
and $K$, the degrees $(l(P):k(P))$ and $(l(C):K)$ are the same.
By transitivity, it follows that $(K:k(P))=1$, as we wanted to show.
\end{proof}

\begin{remark}
The proposition is false in positive characteristic $p$, precisely because of the
purely inseparable extenstion given by the Frobenius $\text{Frob}_p$: if $P$ is a
$K$-rational point of $C$, then so is its Frobenius transform Frob$_p(P)$, which
clearly generates a proper subfield.
\end{remark}

\begin{remark}\label{R:Hur}
By the argument in the proof, $g(C)\geq g(D)$ for any cover $C\to D$ over $k$
(not necessarily algebraically closed). Moreover, if these genera are equal
and $\geq 2$, then $C\to D$ must have degree one, hence is an isomorphism.
In summary, given a cover $C\to D$ of curves of genus at least two,
either $C\iso D$ (and the cover itself is an isomorphism) or $g(C)>g(D)$.
\end{remark}

\begin{remark}\label{R:genpt}
Let $C$ be a curve of genus $g$ at least two and let $K:=k(C)$ be its function
field.
Any non-constant morphism
$C\to C$ is necessarily an automorphism by our previous remark. Hence under
Proposition \ref{P:mor}, we have a one-one correspondence between the $K$-rational generic
points of $C$ and the automorphisms of $C$. In particular, $C(K)\setminus
C(k)$ has cardinality at most $84(g-1)$, as this is the maximum number of
automorphisms of
$C$ (see for instance \cite[IV Exercise 2.5]{H77}). In particular, if $C$ has
no non-trivial automorphism (which is the `generic' case for $g\geq 3$), then $C$
has a unique generic point (in any fixed function field).
\end{remark}

\subsection{General collections of curves}
By a \emph{general collection of curves over $k$}, we mean a countable set
$\mathcal C$ of  ($k$-isomorphism classes of) curves over $k$ 
of genus at least two, such that no two distinct curves in
$\mathcal C$ are isomorphic over $\bar k$. The \emph{function field} $k(\mathcal C)$ of
$\mathcal C$ is by definition  the field generated by all the function
fields of curves in
$\mathcal C$. More precisely, take a universal field $\Omega$ (that is, an
algebraically closed field containing $k$ and of cardinality larger than
any of the fields we use otherwise), and for
each $C\in\mathcal C$, fix a subfield $k_C\subseteq \Omega$ isomorphic to its
function field $k(C)$.  Then
$k(\mathcal C)$ is the subfield of $\Omega$ generated by all the $k_C$ with
$C\in\mathcal C$.  Note that the field generated by the subfields $k_{C_1},\dots,
k_{C_i}$ is isomorphic to the function field of the product variety
$C_1\times\dots\times C_i$, for $C_1,\dots,C_i$
distinct curves in $\mathcal C$. Moreover, $k(\mathcal C)$
is the union of   the  function fields of all such products. In
particular, $k(\mathcal C)$ is well-defined up to isomorphism.

\begin{proposition}\label{P:mem}
Let $\mathcal C$ be a general collection of curves over $k$ and let $k(\mathcal
C)$ be its  function
field. Suppose all curves in $\mathcal C$ have genus at most
$g$ and let $D$ be an arbitrary curve of genus at least $g$. Then the function
field $k(D)$ embeds in $k(\mathcal C)$ if and only if $D$ is isomorphic to 
some curve in $ \mathcal C$.
\end{proposition}
\begin{proof}
Let us write $k_D$ for the image of $k(D)$  in $k(\mathcal C)$. Let
$C_1,C_2,\dots$
be an enumeration of $\mathcal C$, and let $k_i:=k(C_1\times\dots\times
C_i)\subseteq k(\mathcal C)$ be as above, so that in particular $k(\mathcal C)$
is the union of all $k_i$.  Since
$k_D$ is finitely generated, it lies in some $k_i$. Let $i$ be minimal such.
 By Proposition~\ref{P:gen2} and minimality of $i$,  the function field $k_{i-1}(D)$ of $D$ over
$k_{i-1}$ is isomorphic to the field generated by $k_{i-1}$ and  $k_D$, and
hence has transcendence degree one over $k_{i-1}$ by Lemma~\ref{L:ac}.
The extension
$k_{i-1}(D)\to  k_i=k_{i-1}(C_i)$ is finite, since both fields have transcendence
degree one over $k_{i-1}$, and hence by the discussion of Proposition \ref{P:mor},
determines a cover $C_i\to D$ over $k_{i-1}$. If $C_i\not\iso D$, then $g(C_i)>g(D)\geq g$
by Remark~\ref{R:Hur}, contradicting the assumption that $g(C_i)\leq g$.
\end{proof}

We say that a general collection of curves $\mathcal C$ is \emph{non-covering},
if there is no cover relation between any two distinct curves
in $\mathcal C$. Immediately from Remark~\ref{R:Hur} we get:

\begin{lemma}\label{L:genusnoncov}
The collection  of all isomorphism classes of curves
over $k$ of a fixed genus $g\geq 2$ is non-covering.\qed
\end{lemma}

Note that the collection in Lemma~\ref{L:genusnoncov} is in one-one
correspondence with the set of $k$-rational points of the \emph{moduli space}
$\mathcal M_g$ of all curves of genus $g$. Recall that the moduli space has
dimension $3(g-1)$, for   $g\geq 2$.

\begin{theorem}\label{T:noncovcat}
Let $\mathcal C$ be a non-covering collection of curves over $k$, and let
$k(\mathcal C)$ be its function field, as in Proposition \ref{P:mem}.
For any curve $C$ in $\mathcal C$, any embedding of
its function field into $k(\mathcal C)$ has image equal to $k_C$.
In particular, if $P$ is a $k(\mathcal C)$-rational point on $C$ which
is not $k$-rational, then $P$ is $k_C$-rational and the natural inclusion
$k(P)\subseteq k_C$ is an equality.
\end{theorem}
\begin{proof}
Let $K$ be the subfield of $k(\mathcal C)$ generated by all the function
fields $k_D$ with $D\in\mathcal C$ different from $C$.  Since $\mathcal C$
is non-covering, we can apply
Lemma~\ref{L:noncov} over the subfield $k_C\subseteq k(\mathcal C)$.
Since $k(\mathcal C)$ is generated by $K$ and $k_C$,
 Lemma~\ref{L:noncov} shows that $P$ is $k_C$-rational. As $P$ is not
$k$-rational, Proposition~\ref{P:gen2} then shows that $k(P)=k_C$, as we wanted
to show. To prove the first assertion, let $k(C)\to k(\mathcal C)$ be an
embedding, and let $P$ be the corresponding generic point of $C$ defined over $k(\mathcal
C)$. In particular, $k(P)$ is the image of the above embedding, and we just
argued that $k(P)=k_C$.
\end{proof}

\begin{corollary}
Let $\mathcal C$ be a non-covering collection of curves   over
$k$, and let  $k(\mathcal C)$ be its function field. Then the Galois group $G$
of $k(\mathcal C)$ over $k$ is equal to the direct product of the automorphism
groups of each curve in $\mathcal C$. In particular, if all curves in
$\mathcal C$ have genus at most $g$, then $G$ has exponent at most $84(g-1)$.
\end{corollary}
\begin{proof}
Let $C$ be a curve in $\mathcal C$ with function field $k_C \subseteq k(\mathcal C)$,
and let $g\in G$. Since $g(k_C)\iso k_C$, it must be equal
to it by Theorem~\ref{T:noncovcat}. In other words, the restriction of $g$ to
$k_C$ belongs to   the Galois group $H_C$ of   $k_C$ over $k$. In particular,
the restriction map induces a split group homomorphism $G\to H_C$. Since $k(\mathcal
C)$ is generated by all the $k_C$ with $C\in\mathcal C$, it follows that $G$
is isomorphic to the direct product of all $H_C$. Since $H_C\iso
\operatorname{Aut}_k(C)$ has order at most $84(g(C)-1)$, the result follows
(see Remark \ref{R:genpt}).
\end{proof}

\section{Construction of a computably categorical field of infinite
transcendence degree}
\label{sec:ccfields}

In this section, $Q$ is a (countable) computable field
of characteristic zero. Given a collection
$\C=\{ C_0,C_1,\ldots\}$ of curves $C_i$ over $Q$,
we say that $\C$ \emph{has the effective Mordell property}
(or \emph{is effectively Mordell}) if the values
$|C_i(Q)| < \infty$ are computable as a function of $i$.
Of course, this depends not only on $\C$, but also on the
ordering of the curves in $\C$.  Normally we assume a fixed
computable listing of these curves:  some computable function
$g(i)$ gives the defining equation of each $C_i$ over $Q$.
One might hope to compute each $|C_i(Q)|$ effectively
while still allowing $|C_i(Q)|=\aleph_0$,
but in practice we will deal with collections $\C$
with the basic \emph{Mordell property}, i.e., for
which all $|C_i(Q)|$ are finite.  In this case,
with a computable ground field $Q$ and a computable
listing of the curves in $\C$,
the effective Mordell property for $\C$ means
that we can computably determine
the entire set $C_i(Q)$ of solutions for each $i$, since
to do so we only need to be able to compute their number,
using the effectiveness of the Mordell property,
and then search for them all.  
Note that by Faltings's positive solution of the Mordell
Conjecture, each curve of genus two has only finitely many
solutions over a number field.

For our construction, we need a non-covering collection of curves
over the ground field $Q=\Q$ with the effective Mordell property.
For instance, we might in light of Lemma \ref{L:genusnoncov}
ask whether there are infinitely many curves of a fixed genus for which
we can effectively determine their $\mathbb Q$-rational points.
After all, a ``generic'' choice will produce a curve without any
$\mathbb Q$-rational points, a trivial instance of effective Mordell.
Although we do not know the answer to this, we can prove:

\begin{theorem}
\label{thm:Fermatcurves}
There exists a non-covering collection $\mathcal C$ of curves
over $\mathbb Q$ with the effective Mordell property.
\end{theorem}
\begin{proof}
Our collection will consist of \emph{Fermat curves},
which is to say, plane curves $D_N$ with affine equation
$d_N(X,Y)=X^N+Y^N-1$. By Wiles's positive solution of
Fermat's Theorem, the set of
Fermat curves $D_N$ with $N\geq 3$ is effectively Mordell, since
the only $\mathbb Q$-rational points of such a curve
are $(1,0)$ and $(0,1)$ (and the point at infinity $(1:-1:0)$).
Moreover, the genus of $D_N$ is equal to $(N-1)(N-2)/2$.
So it remains to find a non-covering subset. To this end, we
will choose inductively a set of prime
exponents $p_0,p_1, \dots$ as follows. Let $p_0=5$, and choose
$p_{i+1}$ to be the least prime bigger than $(4(p_i-1)(p_i-2))^2$.
Now $D_{p_i}$ does not cover $D_{p_j}$ for $i>j$
by Lemma \ref{L:nocovFermat} below, and for $i<j$
by Remark \ref{R:Hur}.
\end{proof}

To prove the  lemma,  we need the following result from \cite[Lemma 9.3]{BGGP05},
using the explicit value $\text{gcd}(m!,N)$ for $N'$ which they give there.
\begin{lemma}[Baker, Gonz\'alez-Jim\'enez, Gonz\'alez \& Poonen]
\label{L:BGGP05}
Let $C$ be a curve of genus $g\geq 2$ and suppose $m>180$ has the property that $\phi(n)>8g$ for all $n>m$, where $\phi$ is  Euler's totient function. If $C$ is dominated by some Fermat curve $D_N$, then it is already dominated by the Fermat curve $D_{N'}$ where $N'=\text{gcd}(m!,N)$.
\end{lemma}

\begin{lemma}\label{L:nocovFermat}
Let $C$ be a curve of genus $g\geq 2$ and let $D_p$ be
the Fermat curve of prime exponent $p$.
If $p>64g^2$, then there is no cover relation between $C$ and $D_p$.
\end{lemma}
\begin{proof}
Since $D_p$ has genus
$(p-1)(p-2)/2>g$, it cannot be covered by $C$ by Remark~\ref{R:Hur}. We verify that $m=64g^2$ satisfies the hypotheses of Lemma~\ref{L:BGGP05}. Let $n>64g^2$. 
By the well-known estimate for the Euler's totient function
$\phi(n)\geq \sqrt n$, we get $\phi(n)> 8g$. Since $p$ does not divide $(64 g^2)!$, there can be no covering $D_p\to C$.
\end{proof}

Fixing the sequence of primes $\langle p_s\rangle_{s\geq 0}$
defined in Theorem \ref{thm:Fermatcurves}, we begin with
the purely transcendental extension $\Q(x_0,x_1,\ldots)$,
and for each $s$, adjoin an element $y_s$ satisfying
the \mbox{$p_s$-th} Fermat curve $d_{p_s}(x_s,y_s)=x_s^{p_s}+y_s^{p_s}-1=0$.
We write $F_0=\Q$ and $F_{s+1}=F_s(x_s)[y_s]/(d_{p_s}(x_s,y_s))$,
so $F_{s+1}$ has transcendence degree $1$ over $F_s$,
and we may take each $F_s$ to be computable within $F$,
uniformly in $s$.
Clearly $F$ has infinite transcendence degree over $\Q$.

\begin{theorem}
\label{thm:Fermat}
The computable field $F$ built above is computably categorical
and has infinite transcendence degree.
\end{theorem}
\begin{proof}
Let $\Ftilde$ be any computable field isomorphic to $F$,
say via a noncomputable isomorphism $\psi:F\to\Ftilde$.  We
define our computable isomorphism $f$ on each $F_s$ in turn,
starting with the unique (and computable) embedding of the
rationals $F_0$ onto the prime subfield $\Ftilde_0$ of $\Ftilde$.
%$\Qtilde$ of $\Ftilde$.
Given $f\res F_{s-1}$, search for any nonzero elements
$\xtilde$ and $\ytilde$ in $\Ftilde$ such that $d_{p_s}(\xtilde,\ytilde)=0$ there.
We must find them, since $F\cong\Ftilde$, and 
$(\tilde x,\tilde y)$ is then a generic point of $D_p$ over $\Q$.
So we define
$f(x_{s})=\xtilde$ and $f(y_{s})=\ytilde$, then extend this $f$
to the rest of $F_{s}$, which is generated over $F_{s-1}$
by these elements.  This gives a field embedding of $F_s$
into $\Ftilde$, whose image is shown by Theorem \ref{T:noncovcat}
to be uniquely determined as the subfield $\psi(F_s)$
(although the embedding itself may not equal $\psi$ on $F_s$).
So the union of these embeddings (for all $s$) must
map $F$ onto $\Ftilde$.
\end{proof}

One could begin with any prime $>3$ as $p_0$,
thereby building countably many non-isomorphic
computable, computably categorical fields of infinite
transcendence degree.  (Our use of Proposition \ref{P:gen2}
requires that $p_0\neq 3$, since the curves must
all have genus $\geq 2$.)  And by varying the choice of
the subsequent primes $p_1,p_2,\ldots$, one could get uncountably
many relatively computably categorical fields of infinite
transcendence degree, although of course only
countably many of them would be computably presentable.

Next we adapt Theorem \ref{thm:Fermat} to a more general setting.
\begin{theorem}
\label{thm:general}
Let $k$ be any finitely generated field of characteristic $0$,
and $\C$ any computably enumerable collection
of curves over $k$, all of genus $\geq 2$,
such that $\C$ is non-covering and effectively Mordell over $k$.  (In particular,
$C(k)$ is finite for all $C\in\C$.)
Then the function field $k(\C)$ is computably categorical.
\end{theorem}
Saying that $\C$ is computably enumerable (or \emph{c.e.}) means
that there is a computable function $g$ such that for every $s$,
$g(s)\in k[X,Y]$ is a polynomial defining a curve $C_s$,
and $C_0,C_1,\ldots$ is a list of all curves in $\C$
without repetitions (even up to isomorphism).
Below we write $g_s$ for $g(s)$.
\begin{pf}
First, since $k$ is finitely generated, it has a computable presentation.
Then, by enumerating the curves in $\C$ and adjoining to $k$
a generic solution for each, we may build a computable field which,
by Theorem \ref{T:noncovcat}, is isomorphic to $k(\C)$.
Thus $k(\C)$ is computably presentable.

Now let $F$ and $\Ftilde$ be computable fields isomorphic
to $k(\C)$.  The finitely generated field $k$ is
computably enumerable within $F$, and
for a fixed (not necessarily computable) isomorphism
$\alpha:F\to\Ftilde$, the image $\ktilde=\alpha(k)$
will likewise be c.e.\ within $\Ftilde$.
Moreover, the isomorphism $f_0=\alpha\res k$
from the subfield $F_0=k$ onto $\Ftilde_0=\ktilde$
is computable:  we need only know the images
under $\alpha$ of the finitely many generators of $k$.

As in Theorem \ref{thm:Fermat}, the construction is
straightforward.  For each $s$, using
the effective Mordell property, we can compute the
(finite) number $j$ of distinct solutions of $g_s$ in $k$.
Search until we have found all $j$ solutions of $g_s$
in $k$ and one additional solution $(x_s,y_s)$ of $g_s$ in $F$.
Then we search in $\ktilde$ for $j$ distinct solutions to
the polynomial $\gtilde_s (X,Y)\in\ktilde[X,Y]$
whose coefficients are the images of those of $g_s$ under $f_0$, and in $\Ftilde$
for one additional solution $(\xtilde_s,\ytilde_s)$ to $\gtilde_s$.
We must find such a pair $(\xtilde_s,\ytilde_s)$, and we set
$f(x_s)=\xtilde_s$ and $f(y_s)=\ytilde_s$.
It is then easy to compute $f$ on all of $F$,
since Theorem \ref{T:noncovcat} shows $F$
to be generated by $\set{x_s,y_s}{s\in\omega}$.
For any $z\in F$, we compute $f(z)$ just by searching
for an $n$ and an $h\in k(X_1,\ldots,X_n,Y_1,\ldots,Y_n)$ with
$h(x_1,\ldots,x_n,y_1,\ldots,y_n)=z$, since then
$f(z)=h(\xtilde_1,\ldots,\xtilde_n,\ytilde_1,\ldots,\ytilde_n)$.

To see that this $f$ really is an isomorphism,
we appeal to the following lemma.
In Theorem \ref{thm:Fermat}, the corresponding
fact is immediate (and appears explicitly in the discussion
of the proof of Proposition \ref{prop:solns} below).
Now in this more general
situation, we make sure that it is safe to choose
an arbitrary solution
$(\xtilde_s,\ytilde_s)$ of $\gtilde_s$ in $(\Ftilde-\ktilde)$.

\begin{lemma}
\label{lemma:ext}
For every $s$ and any pairs $(x,y), (x',y')\in (F_{s+1}-F_s)^2$
with $g_s(x,y)=g_s(x',y')=0$,
every automorphism $\psi$ of $F_s$ which fixes
$k$ pointwise extends to a unique automorphism of $F_{s+1}$
which maps $x$ to $x'$ and $y$ to $y'$.
\end{lemma}
\begin{pf}
By Lemma \ref{L:ac}, $F_s$ is algebraically closed
within $F_s(C_s)=F_{s+1}$, so each of $x$ and $x'$ must be transcendental
over $F_s$.  Therefore $\psi$ extends to a partial automorphism
$\psi'$ of $F_{s+1}$ by mapping $x$ to $x'$.
Next, by absolute irreducibility, $g_s(x,Y)$ is irreducible
in $F_s(x)[Y]$, and likewise for $x'$, so
$$F_s(x,y)\cong F_s(x)[Y]/(g_s(x,Y)) \cong F_s(x')[Y]/(g_s(x',Y)) \cong F_s(x',y'),$$
with the middle isomorphism being induced by $\psi'$
on the quotients of the polynomial rings.
(It is important here that $\psi'$ fixes the coefficients of $g_s(X,Y)$.)
But by Theorem \ref{T:noncovcat}, $x$ and $y$
together generate $F_{s+1}$ over $F_s$, as do $x'$ and $y'$,
which makes it clear both that we have an automorphism of $F_{s+1}$,
and that it is unique.
\qed\end{pf}

Our claim that the $f=\cup_s f_s$ constructed above is an isomorphism
will follow from $(f^{-1}\circ\alpha)$ being an automorphism of $F$,
where $\alpha$ was the given isomorphism from $F$ to $\Ftilde$.
Since $f_0=\alpha\res k$, we know that $(f_0^{-1}\circ\alpha)$ is the identity
on $k=F_0$.  For each $s$, $(f^{-1}\circ\alpha)\res F_{s+1}$
is just the unique extension of $\psi=(f^{-1}\circ\alpha)\res F_s$
given by the lemma,
with $x=\alpha^{-1}(\xtilde_s)$, $y=\alpha^{-1}(\ytilde_s)$,
$x'= x_s$, and $y'=y_s$.
\qed\end{pf}

We offer a criterion for computable enumerability
of certain collections $\C$ of curves.
\begin{theorem}
\label{thm:cecurves}%
Let $g\geq 2$,  let $Q$ be a computable field, and let $\mathcal C$ be an
infinite subset (indexed by natural numbers)
of the $Q$-rational points  of the moduli space $\mathcal M_g$. Suppose for
each index of a $Q$-rational
point of the moduli space $\mathcal M_g$, we can effectively compute a
defining set of equations over $Q$ of an affine model of the curve of genus
$g$ determined by this $Q$-rational point. Then  relative to a
computable representation of the field extension $Q\subseteq Q(\mathcal C)$,
the subset $\mathcal C$ is computably enumerable inside $\mathcal M_g(Q)$.
\end{theorem}
\begin{proof}
Let $C$ be a $Q$-rational point of
$\mathcal M_g$, viewed as a curve over $Q$ of genus $g$, and let
$f_1,\dots,f_s\in Q[x_1,\dots,x_m]$ be the defining
equations of an affine model of $C$ (meaning that
the function field of $C$ is the field of fractions of
$Q[x_1,\dots,x_m]/(f_1,\dots,f_s)Q[x_1,\dots,x_m]$).
We now search for a solution $P\in Q(\mathcal C)^m$ of the system of
equations $f_1=\dots=f_s=0$  which is not defined over the computable subfield
$Q$.
If such a solution exists, then $C\in\mathcal C$ by Proposition~\ref{P:mem}.
\end{proof}

\begin{ex}
For instance, if $g=2$, any curve $C$ of genus $2$ has an affine model $C_f$
given by an equation $y^2=f(x)$ with $f$ a polynomial of degree $6$ without
double roots (so that in particular  $C_f$ is non-singular whence isomorphic to an
open subset of $C$).
Moreover, two such models $C_{f_1}$ and $C_{f_2}$ are
birational to the same curve if and only if $f_1$ and $f_2$ are equal up to a
fractional linear transformation, that is to say, if and only if
$$
f_1(x)=(cx+d)^6\cdot f_2\left(\frac{ax+b}{cx+d}\right)
$$
for some $a,b,c,d\in Q$ with $ad-bc\neq 0$ (see for instance \cite[\S1.1]{CF96}).
Put differently, the subset $\mathbb H$ in the Hilbert scheme $\mathbb P_Q^7$ of
degree six polynomials without a double root admits a natural  action of the
group of fractional linear transformations $\Gamma:=\text{PGL}(2,Q)$, and the
geometric quotient $\mathbb H//\Gamma$ is the moduli space $\mathcal M_2$.
\end{ex}

\section{An Intrinsically Computable Transcendence Basis}
\label{sec:basis}

The field $F$ built in Theorem \ref{thm:Fermat}
has uncountably many automorphisms, as one sees
by defining, for any $S\subseteq\omega$, an automorphism
extending the map
$$ x_i \mapsto\left\{\begin{array}{cl}
x_i, & \text{if~}i\in S;\\
y_i, & \text{if~}i\notin S.
\end{array}\right.$$
Therefore, between $F$ and any computable field $\Ftilde\cong F$,
there exist $2^\omega$-many isomorphisms.  The point of
Theorem \ref{thm:Fermat} was that at least one must
be computable, but the image of the transcendence basis
$B=\set{x_i}{i\in\omega}$ can have arbitrarily high Turing degree
as well:  its image under the above automorphism is Turing-equivalent
to $S$.  In the language of computable model theory,
this says that $B$ is far from being \emph{intrinsically
computable}.
\begin{defn}
\label{defn:intrinsicallycomputable}
Let $\M$ be any computable structure, and $R$ an $n$-ary
relation on $\M$ (generally not in the signature of $\M$).
The \emph{degree spectrum of $R$ on $\M$} is the set
$$ \DS{\M}{R}=\set{\deg(f(R))}
{\B\cong\M\text{~via~}f~\&~\B\text{~is a computable structure}}.$$
If $\DS{\M}{R}=\{\bfz\}$, then $R$
is \emph{intrinsically computable}.
\end{defn}
So the degree spectrum measures the amount of complexity
that can be added to (or withheld from) $R$ under isomorphisms
onto other computable structures.  If that amount is bounded
above (or below), then we think of $R$ as being intrinsically
no more complex than (or no less complex than) that bound.
Degree spectra of relations have been studied widely in computable
model theory; see \cite{H98} for a survey of results.  With our field $F$,
the degree spectrum of the basis $B$, viewed as a unary relation
on $F$, contains every Turing degree, simply by the argument
above using automorphisms from $F$ to itself.

In this section we show that we can go to the opposite extreme:
an infinite transcendence basis for a computable field can be
intrinsically computable.  This would be trivial if the basis were finite,
of course, but for the infinite case, we believe this is the
first proof of Theorem \ref{thm:intrinsicallycomputable}, and more
generally the first proof of Corollary \ref{cor:intrinsicallycomputable}
(except insofar as the corollary already follows from the proof of Theorem \ref{thm:Fermat}:
every computably categorical field with an infinite computable transcendence basis
instantiates the corollary).

\begin{theorem}
\label{thm:intrinsicallycomputable}
There exists a computable field $F$, of infinite transcendence degree
over its prime field $\Q$, with an intrinsically computable
transcendence basis.
\end{theorem}
\begin{corollary}
\label{cor:intrinsicallycomputable}.
There exists a computable field $F$, of infinite transcendence degree
over its prime field $\Q$, such that every computable field isomorphic
to $F$ has a computable transcendence basis.
\end{corollary}

The field $F$ in the theorem
will be the computable field defined in Theorem \ref{thm:Fermat},
and the key to the proof is knowing the exact number of solutions
in $F$ to each of the Fermat polynomials $d_{p_i}(X,Y)$
used to identify transcendence basis elements.  Therefore,
we need the following essential result of Tzermias and Leopoldt.

\begin{thm}[Tzermias \cite{T95}; Leopoldt \cite{L96}]
\label{thm:Tzermias}
Over an algebraically closed field $K$
of characteristic $0$, the automorphism
group of the projective curve $X^p+Y^p+Z^p$
is the semidirect product of the symmetric group
$S_3$ and the group $(\mu(p))^2$, where $\mu(p)$
is the multiplicative group of $p$-th roots
of unity in $K$.
\qed\end{thm}

In particular, $S_3$ acts by permuting the coordinates
in $\mathbb{P}^2(K)$, and $(\alpha,\beta)\in(\mu(p))^2$
maps $(x:y:z)$ to $(\alpha x:\beta y:z)$.  Clearly
each of these operations maps solutions of $X^p+Y^p+Z^p=0$
to other solutions; the content of the theorem is
that there are no other automorphisms.

\begin{prop}
\label{prop:solns}
The function field $\Q(D_p)$ of the Fermat curve
$D_p$ of prime exponent $p>3$ contains exactly eight solutions
(given below) to its defining equation $d_p(X,Y)=0$.
\end{prop}
\begin{pf}
Since $\mathbb Q$ contains no non-trivial $p$-th roots of unity,
the automorphism group of $D_p$ is isomorphic to $S_3$,
and corresponds by Remark \ref{R:genpt} to the six generic
points on $D_p$. By Wiles' positive solution to Fermat's Last Theorem,
$D_p(\mathbb Q)$ consists of three points, including the point
at infinity $(1:-1:0)$.  Excluding the point at infinity gives a total of eight affine
solutions to the equation $d_p=0$. 
\qed\end{pf}

We now describe this proof more explicitly for readers
with less background in algebraic geometry.
The function field $\Q(D_p)$ must contain at least one
nontrivial solution $(x,y)$ to $X^p+Y^p=1$.  Now with $p>3$, $D_p$ has
genus $\geq 2$, so if $(x',y')\in(\Q(D_p))^2$
is any nontrivial solution, then
$x'$ and $y'$ generate $\Q(D_p)$, by Proposition \ref{P:gen2}
(as do $x$ and $y$ also).
Hence $x'$ is transcendental over $\Q$, forcing $\Q(x)\cong\Q(x')$,
which in turn yields a chain whose composition is an automorphism of $\Q(D_p)$:
$$\Q(x,y)\cong\Q(x)[Y]/(d_p(x,Y)) \cong \Q(x')[Y]/(d_p(x',Y)) \cong\Q(x',y').$$
Thus $(x':y':-1)$ gives an automorphism of the projective curve
$X^p+Y^p+Z^p$, using Remark \ref{R:genpt},
and so Theorem \ref{thm:Tzermias} shows that the only
projective solutions to $X^p+Y^p+Z^p=0$ in the function field are
$(\alpha x:\beta y:-1)$ and permutations thereof,
with $\alpha$ and $\beta$ ranging over the
$p$-th roots of unity.  But all those roots except $1$
lie outside of $\Q$, since $p$ is odd.
Therefore, any solution to $X^p+Y^p+Z^p=0$ must be of the form
$(x:y:-1)$ or a permutation of this, which translates to an affine solution
$(\lambda x,\lambda y,-\lambda)$ or a permutation.
For our purposes the affine solution must end with $-1$,
in order to yield $X^p+Y^p=1$.
Clearly $(x,y,-1)$ works, and choosing $\lambda=-\frac{1}{x}$
gives $(-1,-\frac{y}{x},\frac1{x})$, while choosing
$\lambda=-\frac1{y}$ gives $(-\frac{x}{y},-1,\frac1{y})$, which
both can be permuted under $S_3$ to the appropriate form.  However,
no other value for $\lambda$ can give the necessary coordinate $-1$.
\begin{cor}
\label{cor:sixsolns}
The six nontrivial solutions of $X^p+Y^p=1$ in $\Q(D_p)$
are $(x,y)$, $(-\frac{y}{x},\frac1{x})$, $(-\frac{x}{y},\frac1{y})$,
and the transpositions of these,
where $(x,y)$ is any one nontrivial solution.
\qed\end{cor}

So in Theorem \ref{thm:Fermat}, when we built
the computable isomorphism $f$ from $F$ onto an
arbitrary computable copy $\Ftilde$, there were actually
six possible images for each $x_s$ in $\Ftilde$,
which we identified by finding nontrivial solutions to $d_{p_s}=0$.
We simply chose $f(x_s)$ to be the first one we recognized,
since Theorem \ref{T:noncovcat} shows that each of the
six is a valid choice.  For Theorem \ref{thm:intrinsicallycomputable},
however, we need a transcendence basis each element
of which has only one possible image in the target field.

\begin{pftitle}{Proof of Theorem \ref{thm:intrinsicallycomputable}}
Let $F$ be the computable field defined in Theorem \ref{thm:Fermat},
presented as
$$ F = \Q(x_0,x_1,\ldots)[y_0,y_1,\ldots]/
(d_{p_i}(x_i,y_i)~:~i\in\omega),$$
where $\la p_i\ra_{i\in\omega}$ was the computable
sequence of primes chosen there.  In the nice presentation $F$
constructed there, the transcendence basis $B=\set{x_i}{i\in\omega}$
was computable, of course, but not intrinsically computable,
as noted above.  However, we can build an intrinsically
computable basis $A$ from it.

The nontrivial solutions of each $d_{p_i}(X,Y)=0$ in $F$
were given by Proposition \ref{prop:solns}:
$\la x_i,y_i\ra$, $\la y_i,x_i\ra$, $\la \frac{1}{y_i},\frac{-x_i}{y_i}\ra$,
$\la \frac{-x_i}{y_i},\frac{1}{y_i}\ra$, $\la \frac{1}{x_i},\frac{-y_i}{x_i}\ra$,
$\la \frac{-y_i}{x_i},\frac{1}{x_i}\ra$, and no others.
For each $i$, let $z_i$ be the sum of these six numbers:
$$
z_i = x_i+y_i + \frac{1}{y_i}  - \frac{x_i}{y_i}+ \frac{1}{x_i} - \frac{y_i}{x_i}
= \left(\frac{1-x_i}{1-x_i^{p_i}}\right)y_i^{p_i-1} +
\left(\frac{x_i-1}{x_i}\right)y_i +\frac{x_i^2+1}{x_i},
$$
where we have used the algebraic relation $y_i^{p_i}=1-x_i^{p_i}$
to rewrite the expression as a polynomial in $y_i$ over $\Q(x_i)$.
Thus $z_i$ clearly lies in $\Q(x_i)[y_i]$ but not in $\Q$,
as $y_i$ is algebraic of degree $p_i$ over $\Q(x_i)$,
and so $z_i$ is algebraically interdependent with $x_i$ in $F$
over $\Q$.  So we may replace each $x_i$ by $z_i$ in the basis $B$,
giving a new transcendence basis $A=\set{z_i}{i\in\omega}$ for $F$.

Now each $z_i$ is defined by a simple existential formula $\psi_i(z)$:
it is the sum of six distinct elements $x\in F$, each satisfying
$$ x\neq 0~\&~x\neq 1~\&~(\exists y\in F) [x^{p_i}+y^{p_i}=1].$$
By Proposition \ref{prop:solns}, this formula $\psi_i(z)$ actually defines
$z_i$ in $F$.  So the transcendence basis $A$ is defined by
$$ (\exists i\in\omega) \psi_i(z).$$
(This definition is in fact a computable infinitary $\Sigma_1$ formula,
quantifying over $\omega$ as well as over $E$,
which is acceptable for us, though less common in model theory.)
Therefore, the image $f(A)$ under any isomorphism $f$
(not necessarily computable) from
$F$ onto any computable field $\Ftilde$ must also be
existentially defined in $\Ftilde$, hence c.e.
Now we invoke a simple lemma to show that $f(A)$
is computable in $\Ftilde$.

\begin{lemma}
\label{lemma:cebasis}
In a computable field $K$, if a transcendence basis
$A$ is computably enumerable, then $A$ is computable.
\end{lemma}
\begin{pf}
$K$ is algebraic over the purely transcendental extension
$Q(A)$, where $Q$ is the prime subfield of $K$.
Given any $t\in K$, we find an $n\in\omega$ and a nonzero
polynomial $p(T)\in Q(a_0,\ldots,a_n)[T]$ with $p(t)=0$,
by enumerating $A=\{ a_0,a_1,\ldots\}$ and searching
though such polynomials (for all $n$ simultaneously).
But then $t$ is algebraic over $Q(a_0,\ldots,a_n)$,
and so $t\in A$ if and only if $t\in\{ a_0,\ldots,a_n\}$.
\qed\end{pf}
This completes the proof of Theorem \ref{thm:intrinsicallycomputable}.
\qed\end{pftitle}

We remark that Proposition \ref{prop:solns} was essential
to this proof of Theorem \ref{thm:intrinsicallycomputable}.
In a field as described more generally in Theorem \ref{thm:general},
one might not be able to compute the number of solutions
to $C_i$ in its function field, and therefore it could be impossible
to state the definitions $\psi_i(z)$ uniformly in $i$.

We also remark that this same proof shows the basis $A$
to be \emph{relatively intrinsically computable}:
its image $f(A)$ under any isomorphism $f$
is computable relative to the Turing degree of
the field $f(F)$.  Likewise, this field $F$, and also
those described by Theorem \ref{thm:general},
are \emph{relatively computably categorical}:
if $\Ftilde$ is a field isomorphic to $F$
but of arbitrary Turing degree, then there is an
isomorphism between these fields which is computable
in the degree of $\Ftilde$.

\section{Questions}
\label{sec:questions}

Several basic questions arise from the algebraic geometry
presented here.  For instance, we do not know whether cover
relations exist among any of the Fermat curves $D_p$ and $D_q$
with $3 < p < q$.  $D_3$ is a special case, of course,
since its genus equals $1$, and in Theorem \ref{thm:Fermatcurves}
we chose specific primes such that
no cover relations exist among those curves.
But the general case remains unclear.

Also, while Theorems \ref{thm:Fermat}
and \ref{thm:general} give examples of
computably categorical fields of infinite transcendence
degree, they leave open the larger question
of determining a structural criterion for
computable categoricity of fields.  If anything, they
make this question appear more difficult, since now
we know that infinite transcendence degree
is not sufficient to rule out computable categoricity.

Specifically, we remarked in Lemma \ref{L:genusnoncov}
that there is no cover relation among any collection
of curves of any single fixed genus $\geq 2$.  Whether
this gives rise to other computably categorical fields
of infinite transcendence degree depends on whether one
can produce a computable collection of such curves with
the effective Mordell property.  It seems plausible that
this can be done, but the proof remains elusive,
and introduces the larger question of computing
the number of solutions in $\Q$ for an arbitrary collection of curves
with the (not necessarily effective) Mordell property.
Of course, the exact number of $\Q$-rational points
on a curve given by a polynomial $q\in\Q[X,Y]$
defines a limitwise monotonic function
$\lim_s g(q,s)$ (that is, with $g$ computable
and with $g(q,s)\leq g(q,s+1)$ for all $q$ and $s$),
and when we restrict to curves of genus $\geq 2$,
this function is total, i.e., the limit is always finite.
However, it remains unknown whether the exact number
is computable from $q$ or not.

We regard this as a natural and challenging question.
It seems related to Hilbert's Tenth Problem for $\Q$,
which demands a decision procedure for the existence
of $\Q$-rational solutions to polynomials in arbitrarily
many variables over $\Q$.  In particular,
an oracle for Hilbert's Tenth Problem for $\Q$
would allow one to determine the exact number of rational
solutions to any curve of genus $\geq 2$:
one can express the question ``does this curve have a solution
distinct from the finitely many I have already found'' in a uniform way
such that the oracle could answer that question.
This would not necessarily work for curves of
genus $<2$, or for varieties in general,
which could contain infinitely many rational points.
(To decide which varieties have the Mordell property appears
to require the jump of the oracle set for Hilbert's Tenth Problem for $\Q$.)
So, even if the effective Mordell property holds for all curves
of genus $\geq 2$, we would still not have an obvious
algorithm for Hilbert's Tenth Problem for polynomials
defining varieties of higher dimension.

\parbox{4.7in}{
{\sc
\noindent
Department of Mathematics \hfill \\
\hspace*{.1in}  Queens College -- C.U.N.Y. \hfill \\
\hspace*{.2in}  65-30 Kissena Blvd. \hfill \\
\hspace*{.3in}  Flushing, New York  11367 U.S.A. \hfill \\
Ph.D. Programs in Mathematics \& Computer Science \hfill \\
\hspace*{.1in}  C.U.N.Y.\ Graduate Center\hfill \\
\hspace*{.2in}  365 Fifth Avenue \hfill \\
\hspace*{.3in}  New York, New York  10016 U.S.A. \hfill}\\
\medskip
\hspace*{.045in} {\it E-mail: }
\texttt{Russell.Miller\at {qc.cuny.edu} }\hfill \\
}

\parbox{4.7in}{
{\sc
\noindent
Department of Mathematics \hfill \\
\hspace*{.1in}  New York City College of Technology \hfill \\
\hspace*{.2in}  300 Jay Street \hfill \\
\hspace*{.3in}  Brooklyn, New York  11201 U.S.A. \hfill \\
Ph.D. Program in Mathematics \hfill \\
\hspace*{.1in}  C.U.N.Y.\ Graduate Center\hfill \\
\hspace*{.2in}  365 Fifth Avenue \hfill \\
\hspace*{.3in}  New York, New York  10016 U.S.A. \hfill}\\
\medskip
\hspace*{.045in} {\it E-mail: }
\texttt{hschoutens\at {citytech.cuny.edu} }\hfill \\
}

\end{document}